\documentclass[a4paper,12pt]{article}
\usepackage{amsmath}
\usepackage{amssymb}
\usepackage{amsthm}
\usepackage[english]{babel}
\usepackage{hyperref}

\newtheorem{theor}{Theorem}
\newtheorem{lemma}[theor]{Lemma}
\newtheorem{prop}[theor]{Proposition}

\theoremstyle{definition}
\newtheorem{defi}[theor]{Definition}
\newtheorem{rem}[theor]{Remark}

\begin{document}

\begin{center}
{\Large 
On the number of non-zero character values in generalized blocks of symmetric groups}

\vspace{12pt}

Lucia Morotti
\end{center}

\begin{abstract}
Given a generalized $e$-block $B$ of a symmetric group and an $e$-regular conjugacy class $C$, we study the number of irreducible characters in $B$ which do not vanish on $C$ and find lower bounds for it.
\end{abstract}

\section{Introduction}

Let $n\geq 1$ and $e\geq 2$. In \cite{kor} K\"{u}lshammer, Olsson and Robinson defined generalized $e$-blocks for symmetric groups, showing that the set of irreducible characters indexed by partitions with the same $e$-core have certain similar properties to the irreducible characters contained in the same $p$-block (for $p$ prime). For a prime $p$ generalized $p$-blocks are equal to $p$-blocks coming from modular representation theory (Nakayama conjecture).

Let $\lambda$ be a partition. We say that $\lambda$ is an $e$-class regular partition if none of its parts is divisible by $e$. In this paper we study lower bounds for the number of irreducible characters contained in the same generalized $e$-block which do not vanish on a certain $e$-regular conjugacy class, that is a conjugacy class indexed by an $e$-class regular partition. In Theorem \ref{t1} we give explicit formulas for such lower bounds, which only depends on the $e$-weight of the considered block.

The work presented here was started in connection to the following question from A. Evseev:

``Let $\pi\in S_n$. Does it always hold that the number of irreducible characters of $S_n$ not vanishing on $\pi$ is at most equal to the number of irreducible characters of $C_{S_n}(\pi)$?''

The lower bounds presented here were found while studying the corresponding lower bounds for partitions with no part larger than $e-1$ (which are $e$-class regular partitions), in order to possibly answer the above question by induction on the largest part of the cycle partition of $\pi$.

Studying the number of non-zero elements on (parts of) columns of character tables can also be seen to connect to work on non-vanishing conjugacy classes, that is conjugacy classes on which no irreducible character vanishes (see \cite{nvan} and \cite{inw}). We say that a partition is a non-vanishing partition if it labels a non-vanishing conjugacy class of a symmetric group. It can be easily seen that any non-vanishing partition is of the form $(3^a,2^b,1^c)$ for some $a,b,c\geq 0$, with $b$ even if $n=3a+2b+c\not=2$. One can also bound $b$ and $c$ using the distribution of 2- and 3-cores. However not much more is known about such partitions, even if non-trivial examples, such as $(2^2,1^3)$, $(3,1^4)$ and $(3,2^2)$, exist.

Before stating the main theorem of this paper we need some definitions. In order to simplify notations, we will identify a generalized block with the set of partitions labeling characters belonging to it.

\begin{defi}
For an $e$-core $\mu\vdash n-we$ with $w\geq 0$, define the generalized $e$-block $B_e(\mu,w)$ to be the  set of the partitions of $n$ with $e$-core $\mu$.
\end{defi}

For the definition and properties of $e$-cores, see \cite{o1}. For partitions $\nu,\lambda$ with $|\nu|=|\lambda|=n$ let $\chi^\nu$ be the irreducible character of $S_n$ labeled by $\nu$ and $\chi^\nu_\lambda$ the value $\chi^\nu$ takes on the conjugacy class labeled by $\lambda$.

\begin{defi}
Let $\lambda\vdash n$ and an $e$-core $\mu\vdash n-we$ with $w\geq 0$. We define $c_\mu(\lambda):=|\{\nu\in B_e(\mu,w):\chi^\nu_\lambda\not=0\}|$.
\end{defi}

\begin{theor}\label{t1}
Let $\mu\vdash n-we$ be an $e$-core with $w\geq 0$. Then
\[\min\{c_\mu(\lambda):\lambda\vdash |\mu|+we\mbox{ is }e\mbox{-class regular and }c_\mu(\lambda)\not=0\}=w+1.\]
\end{theor}

Notice that $c_\mu(\lambda)$ could be equal to 0. This happens for example when taking $e=3$, $\mu=(6,4,2)$ and $\lambda=(10,2,1,1,1)$ or when taking $e=4$, $\mu=(2,1)$ and $\lambda=(2^{2w+1},1)$.

The proof of Theorem \ref{t1} will be divided into Propositions \ref{t2} and \ref{t3}, in which it will be proved that such a minimum is at most and at least $w+1$ respectively. Also in the proof of Proposition \ref{t2} an $e$-class regular partition $\lambda$ will be constructed for which $c_\mu(\lambda)=w+1$.

\section{Proof of Theorem \ref{t1}}

In the proof of the next theorem we will construct an explicit $e$-class regular partition $\lambda$ for which $c_\mu(\lambda)=w+1$. This will then prove that the minimum in Theorem \ref{t1} is at most $w+1$.

\begin{prop}\label{t2}
For an $e$-core $\mu\vdash n-we$ with $w\geq 0$ there exists an $e$-class regular partition $\lambda\vdash n$ for which $c_\mu(\lambda)=w+1$. 
\end{prop}

\begin{proof}

Assume first that $\mu=()$. Then $w\geq 1$ as $n\geq 1$. Let $\lambda=(we-1,1)$. This is an $e$-class regular partition as $e\geq 2$. We will now show that
\[|B_e((),w)\cap\{\nu:\chi^\nu_\lambda\not=0\}|=w+1\]
which will prove the theorem in this case. Notice that
\[\{\nu:\chi^\nu_\lambda\not=0\}=\{(we),(1^{we})\}\cup\{(a,2,1^{we-a-2}):2\leq a\leq we-2\}\]
(this follows easily by applying the Murnaghan-Nakayama formula). Clearly $(we),(1^{we})\in B_e((),w)$ and they are distinct, since $we\geq 2$. We will now count how many partitions of the form $(a,2,1^{we-a-2})$ are in $B_e((),w)$. Let $2\leq b\leq e+1$ and $0\leq c\leq e-1$ with $a\equiv b \mod e$ and $we-a-2\equiv c\mod e$. Then $(a,2,1^{we-a-2})_{(e)}=(b,2,1^c)_{(e)}$ are equal (where $\delta_{(e)}$ is the $e$-core of a partition $\delta$). For any partition $\phi$ and any node $(i,j)$ of $\phi$ let $h_{i,j}^\phi$ be the corresponding hook length. Then
\begin{align*}
h_{1,1}^{(b,2,1^c)}&=b+c+1\equiv we-1\mbox{ mod }e,\\
h_{1,2}^{(b,2,1^c)}&=b,\\
h_{1,3}^{(b,2,1^c)}&=b-2<e,\\
h_{2,1}^{(b,2,1^c)}&=c+2,\\
h_{2,2}^{(b,2,1^c)}&=1<e,\\
h_{3,1}^{(b,2,1^c)}&=c<e
\end{align*}
(formulas for $h_{1,3}^{(b,2,1^c)}$ and $h_{3,1}^{(b,2,1^c)}$ holding if respectively $(1,3)$ or $(3,1)$ are nodes of $(b,2,1^c)$). Also as $b+c+2\equiv we$, either both or none of $b$ and $c+2$ are divisible by $e$. In particular if $(a,2,1^{we-a-2})\in B_e((),w)$ then $b$, and so also $a$, is divisible by $e$. In this case it follows easily that $(a,2,1^{we-a-2})\in B_e((),w)$. So, as $e\geq 2$,
\begin{align*}
c_{()}(\lambda)&=2+|\{a:2\leq a\leq we-2\mbox{ and }e|a\}|\\
&=2+|\{a:e\leq a\leq we-e\mbox{ and }e|a\}|\\
&=2+|\{e,2e,\ldots,(w-1)e\}|\\
&=w+1
\end{align*}
and then the theorem holds in this case.

Assume now that $\mu\not=()$. Let $k$ be maximal with $(k,k)\in[\mu]$, the Young diagram of $\mu$, and define $\lambda:=(we+h_{1,1}^\mu,h_{2,2}^\mu,\ldots,h_{k,k}^\mu)$. From the definition of $k$ it follows that $\lambda\vdash |\mu|+we=n$. Also as $k\geq 1$ and $h_{1,1}^\mu,\ldots,h_{k,k}^\mu$ are not divisible by $e$ (as $\mu$ is an $e$-core) we have that $\lambda$ is $e$-class regular. We will now show that $c_\mu(\lambda)=w+1$. Let
\[\psi\in B_e(\mu,w)\cap\{\nu:\chi^\nu_\lambda\not=0\}.\]
Then $[\mu]\subseteq [\psi]$. Let $s$ be maximal with $(s,s)\in[\psi]$. Then $s\geq k$ and $h_{i,i}^\psi\geq h_{i,i}^\mu=\lambda_i$ for $2\leq i\leq k$. Also from Corollary 2.4.9 of \cite{jk} we have that $\lambda\leq (h_{1,1}^\psi,\ldots,h_{s,s}^\psi)$. So $\lambda=(h_{1,1}^\psi,\ldots,h_{s,s}^\psi)$. From $[\mu]\subseteq [\psi]$ and $h_{1,1}^\psi-h_{1,1}^\mu=we=|\psi|-|\mu|$ it then follows that $\psi$ is obtained from $\mu$ by adding nodes only to the first row and the first column. So $\psi=(\mu_1+d,\mu_2,\ldots,\mu_r,1^{we-d})$ where $0\leq d\leq we$ and $r$ is maximal with $\mu_r>0$. As $\chi^\psi_\lambda\not=0$ for each such $\psi$ (Corollary 2.4.8 of \cite{jk}), we have that
\begin{align*}
c_\mu(\lambda)&=|\{(\mu_1+d,\mu_2,\ldots,\mu_r,1^{we-d}):0\leq d\leq we\mbox{ and}\\
&\hspace{24pt}(\mu_1+d,\mu_2,\ldots,\mu_r,1^{we-d})_{(e)}=\mu\}|.
\end{align*}
We will now check when $(\mu_1+d,\mu_2,\ldots,\mu_r,1^{we-d})_{(e)}=\mu$ holds. Let $0\leq f,g\leq q-1$ with $d\equiv f\mod e$ and $we-d\equiv g\mod e$. Then clearly
\[(\mu_1+d,\mu_2,\ldots,\mu_r,1^{we-d})_{(e)}=(\mu_1+f,\mu_2,\ldots,\mu_r,1^g)_{(e)}.\]
If $f=0$ then $g=0$ and $(\mu_1+d,\mu_2,\ldots,\mu_r,1^{we-d})_{(e)}=\mu_{(e)}=\mu$.

So assume now that $f>0$ and let $\varphi:=(\mu_1+f,\mu_2,\ldots,\mu_r,1^g)$. If no hook in the first row has length divisible by $e$ then $(\varphi_{(e)})_1=\varphi_1=\mu_1+f>\mu_1$. In particular in this case $\varphi_{(e)}\not=\mu$. Assume now that $e|h_{1,j}^\varphi$. As $h_{1,\mu_1+1}^\varphi=f<q$ we have that $j\leq\mu_1$. So
\[(\varphi_{(e)})_1=((\varphi\setminus H_{1,j}^\varphi)_{(e)})_1\leq(\varphi\setminus H_{1,j}^\varphi)_1=\left\{\begin{array}{ll}
j-1,&j>\mu_2\\
\mu_2-1,&j\leq\mu_2
\end{array}\right.<\mu_1\]
(where $\varphi\setminus H_{1,j}^\varphi$ is the partition obtained from $\varphi$ by removing the $(1,j)$-hook) and then also in this case $\varphi_{(e)}\not=\mu$.

In particular
\begin{align*}
c_\mu(\lambda)&=|\{(\mu_1+d,\mu_2,\ldots,\mu_r,1^{we-d}):0\leq d\leq we\mbox{ and}\\
&\hspace{24pt}(\mu_1+d,\mu_2,\ldots,\mu_r,1^{we-d})_{(e)}=\mu\}|\\
&=|\{(\mu_1+de,\mu_2,\ldots,\mu_r,1^{we-de}):0\leq d\leq w\}|\\
&=w+1
\end{align*}
and so the theorem holds also in this case.
\end{proof}

We will prove in Proposition \ref{t3} that, for any $e$-class regular partition $\lambda$, either $c_\mu(\lambda)=0$ or $c_\mu(\lambda)\geq w+1$, which will conclude the proof of Theorem \ref{t1}. We start with a lemma which will be used in the proof of the proposition.

\begin{lemma}\label{l1}
Assume that $\gamma_1$, $\gamma_2$, $\delta_1$ and $\delta_2$ are partitions such that $\gamma_i$ can be obtained from $\delta_j$ by removing a hook of leg length $l_{i,j}$ for $1\leq i,j\leq 2$. Further assume that $\gamma_1\not=\gamma_2$, $\delta_1\not=\delta_2$ and $\delta_1$ and $\delta_2$ cannot be obtained one from the other by removing a hook. Then $(-1)^{l_{1,1}+l_{1,2}+l_{2,1}+l_{2,2}}=-1$.
\end{lemma}

\begin{proof}
Choose $\beta$-sets $X_{\gamma_i}$ and $X_{\delta_j}$ all with the same number of elements (for definition and properties of $\beta$-sets see \cite{o1}). Then $|X_{\gamma_1}\setminus X_{\delta_1}|,|X_{\gamma_1}\setminus X_{\delta_2}|=1$ by assumption that $\gamma_1$ can be obtained from $\delta_1$ and $\delta_2$ by removing a hook. So  $|X_{\delta_1}\setminus X_{\delta_2}|\leq 2$. The case $|X_{\delta_1}\setminus X_{\delta_2}|=0$ is excluded, since by assumption $\delta_1\not=\delta_2$. If $|X_{\delta_1}\setminus X_{\delta_2}|=1$, then one of $\delta_1$ and $\delta_2$ could be obtained from the other by removing a hook, which is also excluded by assumption. So $|X_{\delta_1}\setminus X_{\delta_2}|=2$ and we can write
\begin{align*}
X_{\delta_1}&=\{a_1,a_2,c_1,\ldots,c_k\},\\
X_{\delta_2}&=\{b_1,b_2,c_1,\ldots,c_k\}
\end{align*}
for some $k\geq 0$ and with $a_1,a_2,b_1,b_2,c_1,\ldots,c_k$ pairwise different. From $|X_{\gamma_i}\setminus X_{\delta_j}|=1$ for each $i,j$ and $\gamma_1\not=\gamma_2$, it follows that, up to exchanging $a_1$ and $a_2$ or $b_1$ and $b_2$,
\begin{align*}
X_{\gamma_1}&=\{a_1,b_1,c_1,\ldots,c_k\},\\
X_{\gamma_2}&=\{a_r,b_s,c_1,\ldots,c_k\}
\end{align*}
with $(r,s)\in\{(1,2),(2,1),(2,2)\}$. Notice that $a_2>b_1$ and $b_2>a_1$ as $\gamma_1$ is obtained from $\delta_j$ by removing a hook for $1\leq j\leq 2$. If $(r,s)=(2,2)$ then we would similarly obtain that $a_1>b_2$ and $b_1>a_2$, which would give a contradiction. So $(r,s)\in\{(1,2),(2,1)\}$. Up to exchanging $\delta_1$ and $\delta_2$ we can assume that
\[X_{\gamma_2}=\{a_1,b_2,c_1,\ldots,c_k\}.\]
Then $b_1>a_1$ and $a_2>b_2$. So $a_2>b_1,b_2>a_1$. Up to exchanging $\gamma_1$ and $\gamma_2$ we can assume that $a_2>b_2>b_1>a_1$. So
\begin{align*}
l_{1,1}&=|\{x\in X_{\delta_1}:b_1<x<a_2\}|=|\{c_i:b_1<c_i<a_2\}|,\\
l_{2,1}&=|\{x\in X_{\delta_1}:b_2<x<a_2\}|=|\{c_i:b_2<c_i<a_2\}|,\\
l_{1,2}&=|\{x\in X_{\delta_2}:a_1<x<b_2\}|=|\{c_i:a_1<c_i<b_2\}\cup\{b_1\}|,\\
l_{2,2}&=|\{x\in X_{\delta_2}:a_1<x<b_1\}|=|\{c_i:a_1<c_i<b_1\}|.
\end{align*}
As $a_1,a_2,b_1,b_2,c_1,\ldots,c_k$ are pairwise different we have that
\begin{align*}
l_{2,1}+l_{1,2}&=|\{c_i:b_2<c_i<a_2\}|+|\{c_i:a_1<c_i<b_2\}\cup\{b_1\}|\\
&=|\{c_i:a_1<c_i<a_2\}\cup\{b_1\}|\\
&=|\{c_i:a_1<c_i<a_2\}|+1\\
&=|\{c_i:b_1<c_i<a_2\}|+|\{c_i:a_1<c_i<b_1\}|+1\\
&=l_{1,1}+l_{2,2}+1,
\end{align*}
from which the lemma follows.
\end{proof}

\begin{prop}\label{t3}
Let $\mu\vdash n-we$ be an $e$-core with $w\geq 0$ and $\lambda\vdash n$ be $e$-class regular. Then $c_\mu(\lambda)=0$ or $c_\mu(\lambda)\geq w+1$.
\end{prop}

\begin{proof}
We can assume that $c_\mu(\lambda)\not=0$ and let $\psi\in B_e(\mu,w)$ with $\chi^\psi_\lambda\not=0$. Also let $\varphi$ be obtained from $\psi$ by removing a $ke$-hook for some $k\geq 1$ and define
\[\overline{\chi}^\varphi:=\sum_{\beta\vdash n}d_{\varphi,\beta}\chi^\beta,\]
with $d_{\varphi,\beta}=(-1)^i$ if $\beta$ is obtained from $\varphi$ by adding a $ke$-hook of leg length $i$ and $d_{\varphi,\beta}=0$ otherwise. Notice that $d_{\varphi,\psi}\not=0$. Also if $d_{\varphi,\beta}\not=0$ then $\beta\in B_e(\mu,w)$. By Theorem 21.7 of \cite{j1}, $\overline{\chi}^\varphi_\lambda=0$ since $\lambda$ is $e$-regular and then in particular it does not have any part of length $ke$. As $d_{\varphi,\psi}\chi^\psi_\lambda\not=0$, we can find $\beta(\varphi)\in B_e(\mu,w)$ with $d_{\varphi,\beta(\varphi)}\chi^{\beta(\varphi)}_\lambda\not=0$ and such that $d_{\varphi,\psi}\chi^\psi_\lambda$ and $d_{\varphi,\beta(\varphi)}\chi^{\beta(\varphi)}_\lambda$ have different signs.

We will now show that if $\varphi_1\not=\varphi_2$ are obtained from $\psi$ by removing a $k_1e$-hook and a $k_2e$-hook respectively, then $\beta(\varphi_1)\not=\beta(\varphi_2)$. To do this assume that $d_{\varphi_1,\beta},d_{\varphi_2,\beta}\not=0$ and let $\gamma_i=\varphi_i$, $\delta_1=\psi$ and $\delta_2=\beta$ (notice that then $\delta_1$ and $\delta_2$ satisfy the assumptions of Lemma \ref{l1}, since they are distinct partitions of the same integer). In this case $d_{\varphi_i,\beta}$ and $d_{\varphi_i,\psi}$ are given by $(-1)^{l_i}$ and $(-1)^{r_i}$ respectively, where $l_i$ and $r_i$ are the leg lengths of the hooks being removed from $\beta$ and $\psi$ to obtain $\varphi_i$, so that it follows from Lemma \ref{l1} that $d_{\varphi_1,\beta}d_{\varphi_1,\psi}d_{\varphi_2,\beta}d_{\varphi_2,\psi}=-1$. So, if $\chi^\beta_\lambda\not=0$, either $d_{\varphi_1,\psi}\chi^\psi_\lambda$ and $d_{\varphi_1,\beta}\chi^{\beta}_\lambda$ or $d_{\varphi_2,\psi}\chi^\psi_\lambda$ and $d_{\varphi_2,\beta}\chi^{\beta}_\lambda$ have the same sign. In particular $\beta(\varphi_1)\not=\beta(\varphi_2)$ for $\varphi_1\not=\varphi_2$.

By definition $\psi$ has $w$ hooks of length divisible by $e$ and for each partition $\varphi$ obtained by removing one such hook from $\psi$ we can construct a partition $\beta(\varphi)\in B_e(\mu,w)$ with $\chi^{\beta(\varphi)}_\lambda\not=0$, with the property that the partitions $\beta(\varphi)$ are pairwise different. Also from their definition, the partitions $\beta(\varphi)$ are different from $\psi$. In particular
\[c_\mu(\lambda)=|\{\nu\in B_e(\mu,w):\chi^\nu_\lambda\not=0\}|\geq |\{\psi\}\cup\{\beta(\varphi)\}|=w+1\]
which finishes the proof of the theorem.
\end{proof}

Until now we have only considered $e$-class regular partitions. The two following remarks consider partitions which are not $e$-class regular and show what seems to happens in that case. Also the case $e=1$ is considered.

\begin{rem}\textnormal{
For $e\geq 2$ and for partitions $\lambda$ which are not $e$-class regular it still looks as if either $c_\mu(\lambda)=0$ or $c_\mu(\lambda)\geq w+1$. However, since we do not always have $\overline{\chi}_\lambda^\phi=0$, the proof does not hold any more.
}\end{rem}

\begin{rem}\textnormal{
For $e=1$, let $c(\lambda)=c_{()}(\lambda)$ be the number of irreducible characters of $S_n$ which do not vanish on the conjugacy class labeled by $\lambda$. It can be easily checked that $c((n-1,1))=n-1$, for $n\geq 3$ (since $\chi^\nu_{(n-1,1)}\not=0$ if and only if $\nu\in\{(n),(1^n),(a,2,1^{n-a-2}):2\leq a\leq n-2\}$). In particular since here $w=n$, the previous remark does not hold any more. Here computations seem to show that $c(\lambda)\geq n-1$ for each $\lambda\vdash n$. Since any partition of $n$ has less than $\sqrt{2n}$ different part lengths (as $1+2+\ldots+\lceil \sqrt{2n}\rceil>n$) and as $\chi^{(n)}=1$, one can shows using the same reasoning as in the proof of Proposition \ref{t3} that
\[c(\lambda)\geq  |\{(n)\}\cup\{\beta((n-a)):a\mbox{ is not a part of }\lambda\}|>n-\sqrt{2n}+1.\]
Although we cannot obtain $n-1$ as a lower bound through this argument we can still find a lower bound which is not too far from $n-1$.
}\end{rem}

\section*{Acknowledgements}

The author thanks Anton Evseev for the question which lead to starting this work and Christine Bessenrodt for help reviewing the paper.


\begin{thebibliography}{9}

\bibitem{nvan} S. Dolfi, G. Navarro, E. Pacifici, L. Sanus, P. H. Tiep. {\em Non-vanishing elements of finite groups.}  J. Algebra 322 (2010), 540-545.

\bibitem{inw} I. M. Isaacs, G. Navarro, T. R. Wolf. {\em Finite Group Elements where No Irreducible Character Vanishes.}  J. Algebra 222 (1999), 413-423.

\bibitem{j1} G. D. James. {\em The  Representation  Theory  of  the  Symmetric  Groups.} Springer-Verlag,  Berlin  Heidelberg,  1978.  Lecture  Notes  in  Mathematics, 682.

\bibitem{jk} G. James, A. Kerber. {\em The Representation Theory of the Symmetric Group.} Addison-Wesley Publishing Company, 1981.

\bibitem{kor} B. K\"{u}lshammer, J. B. Olsson, G. R. Robinson. {\em Generalized blocks for symmetric groups.}  Invent. Math. 151 (2003), 513-552.

\bibitem{o1} J. B. Olsson. {\em Combinatorics and Representations of Finite
Groups.} Vorlesungen aus dem Fachbereich Mathematik der Universit\"{a}t GH Essen,
1994. Heft 20.


\end{thebibliography}
\end{document}